\documentclass[12 pt]{article}
\usepackage[utf8]{inputenc}
\usepackage[margin=1.0in]{geometry}
\usepackage{bbm}
\usepackage{amsmath}
\usepackage{amsthm}
\usepackage{enumerate}
\usepackage{mathtools}
\usepackage{amssymb}
\usepackage{tikz}
\usetikzlibrary{shapes,arrows}
\usepackage{pgfplots}
\pgfplotsset{compat=1.14}
\usepackage{pgfplotstable}
\usepackage{lmodern}
\usepackage{etoolbox}
\usepackage[affil-it]{authblk} 
\usepackage{float}
\usepackage{xcolor}
\usepackage{setspace}
\usepackage{url}

\newcommand{\GL}{\mathrm{GL}}
\newcommand{\Mod}[1]{\ (\mathrm{mod}\ #1)}

\newcommand{\icol}[1]{
  \left(\begin{smallmatrix}#1\end{smallmatrix}\right)%
}

\newcommand{\irow}[1]{
  \begin{smallmatrix}(#1)\end{smallmatrix}%
}

\newcommand{\brow}[1]{
  \begin{matrix}(#1)\end{matrix}%
}

\newtheorem{theorem}{Theorem}[section]
\newtheorem{lemma}[theorem]{Lemma}
\newtheorem{proposition}[theorem]{Proposition}

\newtheorem{definition}{Definition}[subsection]
\newtheorem{conjecture}[theorem]{Conjecture}
\newtheorem{remark}[theorem]{Remark}

\makeatletter
\patchcmd{\@maketitle}{\LARGE}{\Large}{}{}
\makeatother

\title{Asymptotics of $p$-torsion subgroup sizes in class groups of monogenized cubic fields}
\author{Mikaeel Yunus}
\date{}

\begin{document}
\maketitle


\begin{abstract}

Bhargava, Hanke, and Shankar have recently shown that the asymptotic average $2$-torsion subgroup size of the family of class groups of monogenized cubic fields with positive and negative discriminants is $3/2$ and $2$, respectively. In this paper, we provide strong computational evidence for these asymptotes. We then develop a pair of novel conjectures that predicts, for $p$ prime, the asymptotic average $p$-torsion subgroup size in class groups of monogenized cubic fields.

\end{abstract}

\section{Introduction}
Asymptotics of class groups of number fields over the rationals have been studied for hundreds of years by a wide variety of mathematicians, including Gauss \cite{G} in the 18th century; Davenport, Heilbronn, Cohen, Lenstra, and Martinet \cite{DH, CL, CM} in the late 20th century; and Bhargava, Fouvry, and Klüners in the 21st century \cite{B, FK}. In 2014, Bhargava and Varma \cite{BV} showed that the asymptotic average $2$-torsion subgroup size of the family of class groups of cubic fields ordered by discriminant remains the same regardless of any local conditions imposed on these cubic fields. In 2018, Ho, Shankar, and Varma \cite{HSV} showed that the same average size remains even when ordering by height rather than discriminant.

However, Bhargava, Hanke, and Shankar \cite{BHS} have recently shown that mandating these cubic fields to have monogenic rings of integers – a \textit{global} condition – suprisingly changes the asymptotic average in question when ordering by height. This result is unexpected and interesting, and little is known about it aside from the calculation of the average $2$-torsion subgroup size in \cite{BHS} (see also \cite{SI} and \cite{SII}).


In this paper, we provide computational evidence that imposing the global monogenicity condition mentioned above does indeed increase the asymptotic average in question. We then provide models that verify the results of \cite{BHS} on the average $2$-torsion subgroup size of the family of class groups of cubic fields with monogenic rings of integers. 

Our general approach is as follows. We first use SageMath \cite{S} to calculate the average $2$-torsion subgroup size of the family of class groups for irreducible, monogenic, and maximal binary cubic forms with height less than $Y$. Here, $Y = 10^{11}$ for positive-discriminant binary cubic forms, and $Y = 10^{10}$ for negative-discriminant binary cubic forms. We then employ genetic programming \cite{K, SL} to predict the necessary asymptotes and provide evidence for their agreement with the theoretical values \cite{BHS}.

We subsequently present our main result: a pair of new conjectures predicting, for prime $p$, the asymptotic averages of $p$-torsion subgroup sizes of class groups of monogenized cubic fields ordered by height.\footnote{Note that ordering these monogenized cubic fields by height instead of discriminant should not change the asymptotic averages in question (compare the results of \cite{B} and \cite{HSV}).} For positive discriminants, we predict the averages size to be:  

\begin{equation}
1+\dfrac{1}{p(p-1)}
\end{equation}

For negative discriminants: 

\begin{equation}
1+\dfrac{1}{p-1}
\end{equation}

We develop these conjectures using similar computational methods to the ones we use to provide evidence for the results of \cite{BHS}. 

\section{Theoretical Background}

Here, we present definitions that are necessary for the result of \cite{BHS} that we work with in this paper.  

\subsection{Preliminary Definitions}
Our goal is to study class groups\footnote{Note that we will be using the same definition of a class group (also known as an ideal class group) as the one provided in \cite{MB}.} of number fields. We begin by reviewing number fields. 

\begin{definition} Let $K_0$ be a field. Then, the \textbf{degree} of the field extension $K / K_0$ is the dimension of $K$ as a vector space over $K_0$. 
\end{definition}

\begin{definition} A \textbf{number field} $K$ is a field extension of $\mathbb{Q}$ that has a finite degree. 
\end{definition}

Number fields are a generalization of the rational numbers, $\mathbb{Q}$. Just as the rational numbers contain the ring $\mathbb{Z}$, there is an analogous {\em ring of integers} of a number field. 

\begin{definition} A \textbf{ring of integers} $\mathcal{O}_K$ of a number field $K$ is the ring of all $\alpha$ in $K$ for which there exists a minimal polynomial $f$ such that $f(\alpha) = 0$ and $f(x)$ has coefficients in $\mathbb{Z}$. 
\end{definition}


The ring of integers of a number field is always a {\em Dedekind domain}. An important difference between rings of integers and $\mathbb{Z}$ is that the rational integers are a {\em principal ideal domain} (in fact, $\mathbb{Z}$ is a {\em Euclidean domain}). 

Additionally, the ring of integers of a number field is maximal in the sense that it is the maximal {\em order} of a number field $K$. 

\begin{definition} Let $K$ be a number field. Then, an \textbf{order} of $K$ is a subring of $\mathcal{O}_K$ whose fraction field is equal to $K$. 
\end{definition}
While orders may not be Dedekind domains, they are always {\em integral domains}. 

\begin{definition} An integral domain $R$ has \textbf{rank} $n$ if its fraction field is a number field of degree $n$. 
\end{definition}



\begin{definition} \label{monogenic1}
We say an integral domain $R$ of rank $n$ is \textbf{monogenic} if and only if
\[ R = \left\{ \sum_{i=0}^{n-1} a_i \gamma^i \mid a_i \in \mathbbm{Z} \right\} \] for some $\gamma \in R$. In other words, all elements of $R$ can be represented as a sum of integer multiples of the powers of exactly one element. 
\end{definition}


We now recall two definitions from \cite{BHS}: 

\begin{definition} \label{monogenic2}
We define a \textbf{monogenized cubic integral domain} to be a pair $(R,\alpha)$, where $R$ is a cubic integral domain, and $\alpha \in R$ such that $R = \mathbb{Z}[\alpha]$. 
\end{definition}

\begin{remark} \label{remark}

A monogenized cubic integral domain is monogenic by definition.

\end{remark}

\begin{definition} \label{monogenic3}
We define a \textbf{monogenized cubic field} to be a pair $(K,\alpha)$, where $K$ is a cubic field, and $\alpha \in \mathcal{O}_K$ such that $\mathcal{O}_K = \mathbb{Z}[\alpha]$. 
\end{definition}

Finally, we recall a definition from \cite{BHS} that is essential to Theorem \ref{Thm2.2} below. 

\begin{definition}
Let $(R, \alpha)$ and $(R',\alpha')$ be two monogenized cubic integral domains. Then, $(R, \alpha)$ and $(R',\alpha')$ are \textbf{isomorphic} if there exists an isomorphism $R \rightarrow R'$ under which $\alpha$ is mapped to $\alpha + m$ for some $m \in \mathbb{Z}$. 
\end{definition}

\subsection{The Delone-Faddeev Correspondence}

\begin{definition}
A \textbf{binary cubic form} $f$ is a cubic homogeneous polynomial in two variables with integer coefficients, i.e. 
$f(x,y) = f_0x^3 + f_1x^2y + f_2 xy^2 + f_3y^3$ where $f_0, f_1, f_2, f_3 \in \mathbb{Z}$.
\end{definition}
We define the action of $$\GL_2(\mathbb{Z}) = \left\{ \gamma = \left(\begin{matrix} a & b \\ c & d \end{matrix}\right) \ \middle| \  a,b,c,d \in \mathbb{Z}\  \mbox{ and } \det(\gamma) = ad-bc = \pm 1\right\}$$
on the space of all integral binary cubic forms by:
	\begin{align*}
	    \gamma \cdot f(x,y) &\coloneqq \frac{1}{\det(\gamma)} \ f\left(\brow{x&y} \cdot \gamma \right) \\ 
	    &= \frac{1}{ad-bc} \ f\left(\brow{x&y} \cdot \left(\begin{matrix} a & b \\ c & d \\ \end{matrix}\right) \right)
	\end{align*}

For the remainder of this paper, we refer to an integral domain of rank 3 as a cubic integral domain. 
 
\begin{definition}
The \textbf{discriminant} of a binary cubic form is 
$$\emph{Disc}(f) := f_1^2f_2^2 - 4f_0f_2^3-4f_3f_1^3 - 27f_0^2f_3^3 + 18f_0f_1f_2f_3$$
\end{definition}






We now recall the Delone-Faddeev correspondence as given in \cite{BST}.
\begin{theorem}[Delone-Faddeev \cite{DF}]\label{df} There is a natural bijection between the set of $\GL_2(\mathbb{Z})$-equivalence classes of irreducible integral binary cubic forms and the set of isomorphism classes of cubic integral domains. Furthermore, the discriminant of an integral binary cubic form $f$ is equal to the discriminant of the corresponding cubic integral domain $R(f)$. 
\end{theorem}

Although we do not define the discriminant of a cubic integral domain here, we can utilize the above theorem to define $$\textrm{Disc}(R(f)) := \textrm{Disc}(f)$$

The construction of a cubic integral domain from a binary cubic form can be described quite explicitly in the form of a multiplication table. Given a binary cubic form $f(x,y) = f_0x^3 + f_1x^2y + f_2 xy^2 + f_3y^3$, the cubic integral domain associated to $f(x,y)$ under Theorem \ref{df} can be written as
\[ R(f) := \{a+b\cdot\omega+c\cdot\theta \mid a,b,c \in \mathbbm{Z} \} \] 
where:
\begin{eqnarray}
\omega\cdot\theta &=& \theta\cdot\omega = -f_0\cdot f_3\\ \label{omegasquared}
\omega^2 &=& -f_0\cdot f_2 - f_1\cdot\omega + f_0\cdot\theta\\
\theta^2 &=& -f_1\cdot f_3 - f_3\cdot\omega + f_2\cdot\theta
\end{eqnarray}
Now, building on Theorem \ref{df}, we can establish a correspondence between binary cubic forms and fraction fields.

\begin{lemma} \label{Comp Methodology Step 2}
Let $f(x,y) = f_0x^3+f_1x^2y+f_2xy^2+f_3y^3$ be a binary cubic form, and let $R(f)$ be its corresponding integral domain (cf.~Theorem \ref{df}). Then the fraction field $\textrm{Frac}(R(f))$ is precisely the fraction field of the polynomial $g(x) = x^3 + f_1x^2 + f_0f_2x + f_0^2f_3$. 
\end{lemma}

\begin{proof}
If $f(x,y) = f_0x^3+f_1x^2y+f_2xy^2+f_3y^3$, then the basis elements of $R(f)$ are $1, \omega, \theta$, where $\omega, \theta$ are determined via Equations (1), (2), and (3) in the discussion following Theorem \ref{df}. 

This means that
\[
\begin{cases}
\theta = \frac{-f_0f_3}{\omega}  \hspace{40mm} \\
\omega^2 = -f_0f_2 - f_1\omega + f_0\theta   \hspace{13.5mm} \\
\end{cases}
\]
Thus, substituting the first equation above into the second equation, 
$$\omega^2 = -f_0f_2 - f_1\omega - \frac{f_0^2f_3}{\omega}$$
Multiplying both sides by $\omega$ and rearranging, we have
$$\omega^3 + f_1\omega^2 + f_0f_2\omega + f_0^2f_3 = 0$$
Since the solution to the above equation is the second basis element of the cubic integral domain corresponding to the given binary cubic form, the fraction field of this binary cubic form is precisely the fraction field of the following polynomial: 
$$g(x) = x^3 + f_1x^2 + f_0f_2x + f_0^2f_3$$
\end{proof}

\subsection{Special Properties of Binary Cubic Forms}

We start with the following definition: 
\begin{definition}
A \textbf{monogenic} binary cubic form is a binary cubic form that corresponds to a monogenic cubic integral domain through the Delone-Faddeev correspondence. 
\end{definition}

In this subsection, we present two constraints that we are allowed to make on the coefficients of monogenic binary cubic forms. These constraints prove useful in the computational methodology (Section 3). 

We begin with a constraint on the $x^3$-coefficient $f_0$ of an arbitrary monogenic binary cubic form. 

\begin{lemma} \label{f_0}
Given $f_0, f_1, f_2, f_3 \in \mathbb{Z}$, the binary cubic form $f_0x^3 + f_1x^2y + f_2xy^2 + f_3y^3$ is monogenic if and only if we can perform a $\GL_2(\mathbb{Z})$-action on it to achieve another binary cubic form $f_0'x^3 + f_1'x^2y + f_2'xy^2 + f_3'y^3$ with $f_0', f_1', f_2', f_3' \in \mathbb{Z}$ and $f_0' = 1$. 
\end{lemma}
\begin{proof}

First, we write the proof of the reverse direction. Assume that $f_0' = 1$. Then, by the Delone-Faddeev correspondence, the following system of equations is satisfied: 
\begin{eqnarray*}
\omega\cdot\theta &=& \theta\cdot\omega = -f_3'\\
\omega^2 &=& -f_2' - f_1'\cdot\omega + \theta\\
\theta^2 &=& -f_1'\cdot f_3' - f_3'\cdot\omega + f_2'\cdot\theta
\end{eqnarray*}
From the second equation in the system above, we have $\theta = \omega^2 + f_2' + f_1'\omega$. As a result, we have that any element in $R(f)$ can be written as a linear combination of powers of $\omega$: 
\begin{eqnarray*}
a + b\omega + c\theta &=& a + b\omega + c(\omega^2 + f_2' + f_1'\omega) \\
&=& (a+cf_2') + (b+cf_1')\omega + c\omega^2
\end{eqnarray*}
Let $a' = a+cf_2'$, $b' = b+cf_1'$, and $c' = c$. Then, we have $R(f) = \{a' + b'\omega + c'\omega^2 \mid a',b',c' \in \mathbb{Z}\}$. Thus, $R(f)$ is monogenic. 

Now, we write the proof of the forward direction. Since the binary cubic form $f(x,y)$ is monogenic, the corresponding integral domain $R(f)$ is monogenic. For some $\alpha \in R(f)$, we can write $R(f) = \langle 1, \alpha, \alpha^2 \rangle$. 

Since $\alpha$ is a member of a cubic integral domain, there exists a monic polynomial $g(x) = x^3 + g_1x^2 + g_2x + g_3$, where $g_1, \cdots, g_3 \in \mathbb{Z}$, such that $g(\alpha) = 0$. Keeping this in mind, we can redefine $R(f)$ as follows: $$R(f) = \langle 1, \alpha + g_1, \alpha^2 + g_2 \rangle$$

We note that $$(\alpha + g_1)(\alpha^2 + g_2) = \alpha^3 + g_1\alpha^2 + g_2\alpha + g_1g_2 = -g_3 + g_1g_2 \in \mathbb{Z}$$

Thus, the basis $\langle 1, \alpha, \alpha^2 \rangle$ corresponds to a binary cubic form through the Delone-Faddeev correspondence. 

Let $\omega = \alpha + g_1$, and $\theta = \alpha^2 + g_2$. We compute the following: 

$$\omega^2 = (\alpha + g_1)^2 = \alpha^2 + 2\alpha g_1 + g_1^2$$

Note that since $\omega = \alpha + g_1$ and $\theta = \alpha^2 + g_2$, we have $\alpha = \omega - g_1$ and $\alpha^2 = \theta - g_2$. Thus, 
\begin{eqnarray*}
\omega^2 &=& (\theta - g_2) + 2(\omega - g_1) g_1 + g_1^2 \\
&=& \theta - g_2 + 2\omega g_1 - 2g_1^2 + g_1^2 \\
&=& -(g_1^2 + g_2) + 2g_1\omega + \theta
\end{eqnarray*}
In conjunction with the above, \eqref{omegasquared} implies that $f_0 = 1$. 

\end{proof}

For the remainder of this paper, we assume any monogenic binary cubic form has $x^3$-coefficient $1$, because out interest lies in cubic domains, not their bases, and we can choose whichever binary cubic form we want to represent its own $\GL_2(\mathbb{Z})$-equivalence class. 

Before we proceed, we define the following property of monogenic binary cubic forms, which we call ``height." This property lets us not only assign a weak ordering to binary cubic forms, but also a weak ordering to cubic fields.

\begin{definition} \label{height} 
Let $f = x^3 + f_1x^2y + f_2xy^2 + f_3y^3$ be a monogenic binary cubic form, and let $I(f) = f_1^2-3f_2$, $J(f)=-2f_1^3+9f_1f_2-27f_3$. Then the \textbf{height} $H$ of $f$ is as follows: 
$$H(f) := \max\left(|I|^3, \frac{J^2}{4}\right)$$
If $f$ is clear from context, we will simply let $I=I(f)$ and $J=J(f)$. 
\end{definition}

Now, having defined the height invariant, we set a constraint on the \textit{second} coefficient $f_1$ of an arbitrary binary cubic form. But first, we must introduce a little more theory. 

\begin{definition} \label{F(Z)}
We define the subgroup $F(\mathbb{Z}) < \GL_2(\mathbb{Z})$ as follows: 

$$F(\mathbb{Z}) = \left\{\gamma_a = \left(\begin{matrix} 1 & 0 \\ a & 1 \end{matrix}\right) \ \middle| \ a \in \mathbb{Z}\right\}$$

\end{definition}

Note that every $\gamma_a \in F(\mathbb{Z})$ has determinant $1$. Furthermore, note that for some monogenic binary cubic form $f(x,y) = x^3 + f_1x^2y + f_2xy^2 + f_3y^3$ and $\gamma_a \in F(\mathbb{Z})$, $\irow{x&y} \cdot \gamma_a = \icol{x+ay \\ y}$. Thus, $\gamma_a \cdot f(x,y) = f(x+ay,y)$. 

\begin{lemma} \label{monogenic_preservation}
Let $f(x,y)$ be a binary cubic form with $x^3$-coefficient $1$, and let $\gamma_a \in F(\mathbb{Z})$. Then, the $x^3$-coefficient of $\gamma_a \cdot f(x,y)$ is also $1$.
\end{lemma}

\begin{proof}
Let $f(x,y) = x^3 + f_1x^2y + f_2xy^2 + f_3y^3$. In the discussion following Definition \ref{F(Z)}, we show that $\gamma_a \cdot f(x,y) = f(x+ay,y)$. We now compute the following: 
\begin{eqnarray*}
f(x+ay,y) &=& (x+ay)^3+f_1(x+ay)^2y+f_2(x+ay)y^2+f_3y^3 \\
&=& x^3 + 3x^2ay+3xa^2y^2+a^3y^3 + f_1x^2y \\
&& + 2f_1xay^2 + f_1a^2y^3 + f_2xy^2 + f_2ay^3 + f_3 y^3 \\
&=& x^3 + (3a + f_1)x^2y + (3a^2 + 2f_1a + f_2)xy^2 \\ 
&& + (a^3 + f_1a^2 + f_2a + f_3)y^3
\end{eqnarray*}
Upon observation, we can see that the $x^3$-coefficient of this polynomial is still $1$. 
\end{proof}

\begin{lemma} \label{ij_invariance}
Let $f(x,y)$ be a binary cubic form with $x^3$-coefficient $1$, and let $\gamma_a \in F(\mathbb{Z})$. Then $I(\gamma_a \cdot f) = I(f)$ and $J(\gamma_a \cdot f) = J(f)$.  
\end{lemma}

\begin{proof}
Let $f(x,y) = x^3 + f_1x^2y + f_2xy^2 + f_3y^3$. In the proof of Lemma \ref{monogenic_preservation}, we show that $$\gamma_a \cdot f = f(x+ay,y) = x^3 + (3a + f_1)x^2y + (3a^2 + 2f_1a + f_2)xy^2 + (a^3 + f_1a^2 + f_2a + f_3)y^3$$

We can see that the action of $\gamma_a$ sends $f_1$ to $3a + f_1$, $f_2$ to $3a^2 + 2f_1a + f_2$, and $f_3$ to $a^3 + f_1a^2 + f_2a + f_3$. Thus, this action sends $I(f) = f_1^2 - 3f_2$ to the following: $$I(\gamma_a \cdot f) = (3a + f_1)^2 - 3(3a^2 + 2f_1a + f_2) = 9a^2 + 6af_1 + f_1^2 - 9a^2 - 6f_1a - 3f_2$$

Thus, by cancellation, $I(\gamma_a \cdot f) = f_1^2 - 3f_2$, so $I(\gamma_a \cdot f) = I(f)$.

Now, we turn towards $J(f) = -2f_1^3+9f_1f_2-27f_3$. We compute the following: 
\begin{eqnarray*}
J(\gamma_a \cdot f) &=& -2(3a + f_1)^3 + 9(3a + f_1)(3a^2 + 2f_1a + f_2) \\
&&- 27(a^3 + f_1a^2 + f_2a + f_3) \\
&=& -2(27 a^3 + 27 a^2 f_1 + 9 a f_1^2 + f_1^3) \\ 
&&+ 9(9 a^3 + 9 a^2 f_1 + 2 a f_1^2 + 3 a f_2 + f_1 f_2) \\
&&- 27(a^3 + f_1a^2 + f_2a + f_3) \\
&=& -54 a^3 - 54 a^2 f_1 - 18 a f_1^2 - 2 f_1^3 \\
&&+ 81 a^3 + 81 a^2 f_1 + 18 a f_1^2 + 27 a f_2 + 9 f_1 f_2 \\
&&+ -27 a^3 - 27 a^2 f_1 - 27 a f_2 - 27 f_3
\end{eqnarray*}

Thus, by cancellation, $J(\gamma_a \cdot f) = -2f_1^3+9f_1f_2-27f_3$, so $J(\gamma_a \cdot f) = J(f)$. This concludes the proof. 
\end{proof}

\begin{definition}
We define the \textbf{height} of a monogenized cubic integral domain to be equivalent to the height of its corresponding orbit of binary cubic forms.
\end{definition}

\begin{lemma} \label{f_1}
In every $F(\mathbb{Z})$-equivalence class of monogenic binary cubic forms, there exists exactly one binary cubic form $f(x,y) = x^3 + f_1x^2y + f_2xy^2 + f_3y^3$ with $f_1 \in \{-1, 0, 1\}$. Additionally, $f_1 \equiv J(f) \Mod 3$. 
\end{lemma}

\begin{proof}

To prove the first part of the lemma, we will start by demonstrating that the binary cubic form $f(x,y)$, as described in the lemma, exists in any arbitrary $F(\mathbb{Z})$-equivalence class. Then, we will show that only one such binary cubic form exists in this equivalence class. 

We have shown in the proof of Lemma \ref{monogenic_preservation} that under the action of an $F(\mathbb{Z})$-matrix, the $x^2y$-coefficient $f_1$ of $f(x,y)$ becomes $f_1 + 3a$. Now, let $f_1' \in \{-1, 0, 1\}$ such that $f_1' \equiv f_1 \Mod 3$. We want to find an $a \in \mathbb{Z}$ such that $f_1 + 3a = f_1'$. We have $3a = f_1' - f_1$. Thus, $$a = \frac{f_1'-f_1}{3}$$

Since $f_1' \equiv f_1 \Mod 3$ – in other words, $f_1' - f_1 \equiv 0 \Mod 3$ – $a$ is always an integer. Hence, we are able to reduce the original binary cubic form to one whose $x^2y$-coefficient is $-1$, $0$, or $1$. 

Now that we have shown that such a binary cubic form exists, we will demonstrate that only one such binary cubic form exists. Let $f_2' = 3a^2 + 2f_1a + f_2$ and $f_3' = a^3 + f_1a^2 + f_2a + f_3$, where $a = \frac{f_1' - f_1}{3}$ and $f_1'$ is defined as above. Let $g(x,y)$ be the ``reduced" form of $f(x,y)$ – i.e. $g(x,y) = x^3 + f_1'x^2y + f_2'xy^2 + f_3'y^3$. 

By Lemma \ref{ij_invariance}, $I(f)$ and $J(f)$ are invariant in this $F(\mathbb{Z})$-equivalence class. Thus, we can derive the following equations: 

\[
\begin{cases}
I(f) = f_1'^2 - 3f_2' \implies f_2' = \frac{f_1'^2-I(f)}{3} \\
J(f) = -2f_1'^3+9f_1'f_2'-27f_3' \implies f_3' = -\frac{2f_1'^3}{27}+\frac{f_1'f_2'}{9}-\frac{J(f)}{27}
\end{cases}
\]

We can see from the above equations that there is only one possible value for $f_2'$ and $f_3'$. Thus, there is exactly one binary cubic form with $x^2y$-coefficient in the set $\{-1, 0, 1\}$ in every $F(\mathbb{Z})$-equivalence class. 

Finally, we show that $f_1 \equiv J(f) \Mod 3$.
\begin{eqnarray*}
J(f) & \equiv & -2f_1^3+9f_1f_2 - 27f_3 \\
& \equiv & -2f_1^3 \\
& \equiv & -2f_1 \\
& \equiv & f_1 \Mod 3
\end{eqnarray*}
\end{proof}

As before, we can treat any arbitrary monogenic binary cubic form as being in an $F(\mathbb{Z})$-equivalence class with a binary cubic form with $x^2y$-coefficient in the set $\{-1, 0, 1\}$. Thus, we can disregard all binary cubic forms with other $x^2y$-coefficients. 

\begin{lemma} \label{ij_podasip}
Let $f(x,y)$ and $g(x,y)$ be binary cubic forms with $x^3$-coefficient $1$. If $I(f) = I(g)$ and $J(f) = J(g)$, then $f$ and $g$ are $F(\mathbb{Z})$-equivalent. 
\end{lemma}

\begin{proof}
Let $f(x,y) = x^3 + f_1x^2y + f_2xy^2 + f_3y^3$, and let $g(x,y) = x^3 + g_1x^2y + g_2xy^2 + g_3y^3$. Moreover, suppose that $I(f) = I(g)$ and $J(f) = J(g)$. 

According to Lemma \ref{f_1}, $f_1 \equiv J \Mod 3$ and $g_1 \equiv J \Mod 3$ – hence, $f_1 \equiv g_1 \Mod 3$. We can pick an $a \in \mathbb{Z}$ such that $g_1 = 3a + f_1$. 

Since $I(f) = I(g)$, we have 
\begin{eqnarray*}
f_1^2 - 3f_2 &=& g_1^2 - 3g_2 \\
&=& (3a + f_1)^2 - 3g_2 \\
&=& 9a^2 + 6f_1a + f_1^2 - 3g_2
\end{eqnarray*}

After cancelling the $f_1^2$ terms and dividing both sides by $3$, we have 
$$-f_2 = 3a^2 + 2f_1a - g_2$$ 
Rearranging, we get
$$g_2 = 3a^2 + 2f_1a + f_2$$

Since $J(f) = J(g)$, we have 
\begin{eqnarray*}
-2f_1^3 + 9f_1f_2 - 27f_3 &=& -2(3a + f_1)^3 + 9(3a + f_1)(3a^2 + 2f_1a + f_2) - 27g_3\\
&=& -2(27 a^3 + 27 a^2 f_1 + 9 a f_1^2 + f_1^3) \\ 
&&+ 9(9 a^3 + 9 a^2 f_1 + 2 a f_1^2 + 3 a f_2 + f_1 f_2) - 27g_3\\
&=& -54 a^3 - 54 a^2 f_1 - 18 a f_1^2 - 2 f_1^3 \\
&&+ 81 a^3 + 81 a^2 f_1 + 18 a f_1^2 + 27 a f_2 + 9 f_1 f_2 - 27g_3 \\
&=& 27a^3 + 27a^2f_1 + 27af_2 - 27g_3
\end{eqnarray*}

Dividing both sides by $27$ and rearranging, we get 
$$g_3 = a^3 + a^2f_1 + af_2 + f_3$$

Combining the expressions that we have obtained for $g_1$, $g_2$, and $g_3$, we have the following: 

$$g(x,y) = x^3 + (3a + f_1)x^2y + (3a^2 + 2f_1a + f_2)xy^2 + (a^3 + f_1a^2 + f_2a + f_3)y^3$$

In the proof of Lemma \ref{monogenic_preservation}, we show that if $\gamma_a \in F(\mathbb{Z})$ such that $\gamma_a = \left(\begin{smallmatrix} 1 & 0 \\ a & 1 \end{smallmatrix}\right)$, then $$\gamma_a \cdot f(x,y) = x^3 + (3a + f_1)x^2y + (3a^2 + 2f_1a + f_2)xy^2 + (a^3 + f_1a^2 + f_2a + f_3)y^3$$

Thus, $g(x,y) = \gamma_a \cdot f(x,y)$, so $f$ and $g$ are $F(\mathbb{Z})$-equivalent. 

\end{proof}

Finally, we state a property of monogenic binary cubic forms that proves to be crucial in our computational methodology (see Section 3). 

\begin{theorem}[Theorem 2.2 of \cite{BHS}]\label{Thm2.2}
There is a natural bijection between $F(\mathbb{Z})$-equivalence classes of binary cubic forms with $x^3$-coefficient $1$ and isomorphism classes of monogenized cubic integral domains. 
\end{theorem}

\subsection{Average $2$-Torsion Sizes of Class Groups of Cubic Fields}

Before stating Theorem 1.2 of \cite{BHS}, we define the following: 

\begin{definition}
Given a finite abelian group $G$ and a prime number $p$, the $p$-torsion subgroup of G, denoted as $G[p]$, is defined as follows: $$G[p] := \{g \in G \ | \ g^p = 1\}$$ Moreover, the \textbf{p-torsion size} of $G$ is defined to be the cardinality of the $p$-torsion subgroup of $G$. 
\end{definition}

A fact about $p$-torsion subgroups, which follows from Lagrange's Theorem, proves to be useful in our later discussion of the computational methodology:

\begin{lemma} \label{TTS}
Let $G$ be a finite abelian group. For a prime number $p$, $$G[p] = \{g \in G \ | \ g^p = g_0^{|G|} \quad \forall g_0 \in G\}$$
\end{lemma}

Now we have enough theoretical background to state the theorems in question from \cite{B} and \cite{BHS}. First, the theorem from \cite{B}: 

\begin{theorem}[Theorem 5 of \cite{B}]\label{Bthm}
Let us denote the $2$-torsion subgroup of the class group of a cubic field $K$ as $Cl(K)[2]$. Then, 
\begin{itemize}
\item The average size, when ordering by discriminant, of $Cl(K)[2]$ for cubic fields $K$ of positive discriminant is $5/4$. 
\item The average size, when ordering by discriminant, of $Cl(K)[2]$ for cubic fields $K$ of negative discriminant is $3/2$. 
\end{itemize}
\end{theorem}

Note that in \cite{HSV}, the authors prove the same mean values, but when ordering by height. 

Now, the theorem from \cite{BHS}:  

\begin{theorem}[Theorem 1.2 of \cite{BHS}]\label{BHSthm}
Again, let us denote the $2$-torsion subgroup of the class group of a cubic field $K$ as $Cl(K)[2]$. Then, 
\begin{itemize}
\item The average size, when ordering by height, of $Cl(K)[2]$ for cubic fields $K$ having monogenized rings of integers with positive discriminant is 3/2. 
\item The average size, when ordering by height, of $Cl(K)[2]$ for cubic fields $K$ having monogenized rings of integers of negative discriminant is 2. 
\end{itemize}
\end{theorem}

This is the result that we provide computational evidence for in this paper. 

It is surprising that the average $2$-torsion sizes from Theorem \ref{Bthm} change when we mandate that the cubic fields in question have monogenized rings of integers. Based on local behavior, we would expect these average $2$-torsion sizes not to change under the added restriction.

Before we continue, we introduce some notation that can be used to express the above theorem: 

\begin{definition}
For some $Y \in \mathbb{Z}$, $Y \geq 0$, let $\mathcal{F}^+_Y$ denote the set of positive-discriminant, maximal, and monogenized cubic integral domains $(R,\alpha)$ with height less than $Y$. Similarly, let $\mathcal{F}^-_Y$ denote the set of negative-discriminant, maximal, and monogenized cubic integral domains $(R,\alpha)$ with height less than $Y$. 

Let $K$ denote the fraction field $\textrm{Frac}(R)$. Then, for $p$ prime, we define the following:  

\noindent\begin{minipage}{.5\linewidth}
\begin{equation}
\mu_p(\mathcal{F}^{+}(Y)) := \frac{\sum\limits_{(R,\alpha)\in\mathcal{F}^+_Y}|Cl(K)[p]|}{|\mathcal{F}^+_Y|} 
\end{equation}
\end{minipage}%
\begin{minipage}{.5\linewidth}
\begin{equation}
\mu_p(\mathcal{F}^{-}(Y)) := \frac{\sum\limits_{(R,\alpha)\in\mathcal{F}^-_Y}|Cl(K)[p]|}{|\mathcal{F}^-_Y|}
\end{equation}
\end{minipage}

\end{definition}

With this new notation, Theorem \ref{BHSthm} can be rewritten as

\noindent\begin{minipage}{.5\linewidth}
$$\lim_{Y \rightarrow \infty} \mu_2(\mathcal{F}^{+}(Y)) = 3/2$$
\end{minipage}%
\begin{minipage}{.5\linewidth}
$$\lim_{Y \rightarrow \infty} \mu_2(\mathcal{F}^{-}(Y)) = 2$$
\end{minipage}

\section{Computational Methodology}
Our computational methodology for verifying Theorem \ref{BHSthm} is summarized below. Our goal is to compute the averages $\mu_2(\mathcal{F}^{+}(Y))$ and $\mu_2(\mathcal{F}^{-}(Y))$ for very large values of $Y$. 

\begin{itemize}
\item{\textbf{Step 1: Generate all monogenic binary cubic forms of bounded height and positive/negative discriminant}}

We can assume that any monogenic binary cubic form $f(x,y)$ in our computation has $f_0 = 1$ by Lemma \ref{f_0}. Moreover, by Lemma \ref{f_1}, we can assume that $f_1 \in \{-1, 0, 1\}$ and $f_1 \equiv J(f) \Mod 3$. By looping through all of the possible values of $I$ and $J$, we can see that every $(I,J)$ pair corresponds to a single $(f_0,f_1)$ pair (by Lemma \ref{ij_podasip}). This approach requires us to impose bounds on $I$ and $J$ so that the number of binary cubic forms we generate remains finite. Upon calculation, we find that given a positive integer $Y$, for all binary cubic forms $f(x,y)$ such that $H(f) < Y$, $|I| < \sqrt[3]{Y}$ and $|J| < 2\sqrt{Y}$. 

We now construct the following nested loop: We loop from $I=-\lfloor{\sqrt[3]{Y}\rfloor}$ to $I=\lfloor{\sqrt[3]{Y}\rfloor}$, and for each individual value of $I$ in this loop, we loop from $J=-2\lfloor{\sqrt{Y}\rfloor}$ to $J=2\lfloor{\sqrt{Y}\rfloor}$. For any $(I,J)$ pair, we will always have $f_0=1$ and $f_1 = -1$, $0$, or $1$ depending on the value of $J$. According to Definition \ref{height}, we also have 
\[
\begin{cases}
I = f_1^2 - 3f_2 \implies f_2 = \frac{f_1^2-I}{3} \\
J = -2f_1^3+9f_1f_2-27f_3 \implies f_3 = -\frac{2f_1^3}{27}+\frac{f_1f_2}{9}-\frac{J}{27}
\end{cases}
\]

If $f_2, f_3$ are integers, then we have generated a valid binary cubic form; otherwise, the $(I,J)$ pair does not have a corresponding binary cubic form. 

\item \textbf{Step 2: Calculate the fraction field $K$ of the monogenized cubic integral domain corresponding to each monogenic binary cubic form}: Before starting this step, note that we can only construct a correspondence between monogenized cubic integral domains and binary cubic forms with $x^3$-coefficient $1$ and $x^2y$-coefficient in the set $\{-1, 0, 1\}$ because of Theorem \ref{Thm2.2}, together with Remark \ref{remark}. 


We now have a large list of binary cubic forms that we must loop through. Given a specific binary cubic form within our loop, we can find the minimal polynomial for the corresponding fraction field using Lemma \ref{Comp Methodology Step 2}. But since the binary cubic forms we are concerned with are all monogenic, by Lemma \ref{f_0}, we can simplify the minimal polynomial for the fraction field to the following: 
$$g(x) = x^3 + f_1x^2 + f_2x + f_3$$
We can now compute the fraction field corresponding to the given binary cubic form. (Note that if the above polynomial is reducible, the fraction field is no longer a field – thus, the corresponding binary cubic form will be ejected from the computation.)

\item \textbf{Step 3: Ensure that each binary cubic form is maximal}: Within our large binary cubic form loop described in Step 2, we must ensure that every binary cubic form's corresponding cubic integral domain is maximal inside its fraction field. In order to do so, we must compute the discriminant of each binary cubic form and its corresponding fraction field, and then ensure that the two discriminants are equal. This will imply that the cubic integral domain corresponding to each binary cubic form is isomorphic to the ring of integers inside the binary cubic form's fraction field, implying that the cubic integral domain is indeed maximal. Any binary cubic form whose cubic integral domain is \textit{not} maximal will be ejected from the computation – otherwise, it will remain until the end of the computation. 

\item \textbf{Step 4: Find the class group of each fraction field $K$ resulting from Step 2}: We do this using SageMath \cite{S}. 

\item \textbf{Step 5: Calculate the number of 2-torsion elements in class group of each fraction field $K$}: By Lemma \ref{TTS}, we simply have to find all the elements of $$Cl(K)[2] = \{g \in Cl(K) \ | \ g^2 = e =  g_0^{|Cl(K)|}\}$$ where $e$ is the multiplicative identity of $Cl(K)$ and $g_0$ is an arbitrary element of $Cl(K)$. We perform this computation for every surviving fraction field $K$ in the large binary cubic form loop described in Step 2. 

\item \textbf{Step 6: Compute averages $\mu_2(\mathcal{F}^{+}(Y))$ and $\mu_2(\mathcal{F}^{-}(Y))$}
\end{itemize}

Our ultimate goal now is to increase the height restriction $Y$ from Step 1 to a large enough number so that we can predict the values of $\mu_2(\mathcal{F}^{+}(Y))$ and $\mu_2(\mathcal{F}^{-}(Y))$. We were able to reach a $Y$-value of 100 billion for positive-discriminant monogenic binary cubic forms, and 10 billion for negative-discriminant monogenic binary cubic forms. (The limiting factor here was computation time – our longest run took about a week to complete.) 

\section{Computational Data Towards Theorem \ref{BHSthm}}

In this section, we present our computational data on the average $2$-torsion size of class groups of monogenized cubic fields, ordered by height. First, we will present two propositions from \cite{BHS} and compare them with computational data. 

\subsection{Asymptotics on Binary Cubic Form Counts}

The following proposition is equivalent to Proposition 3.2 of \cite{BHS}. 

\begin{proposition} \label{bcf-counts} Let the number of $F(\mathbb{Z})$-equivalence classes of \textbf{positive-discriminant} monogenic binary cubic forms with height less than $Y$ be $N^+(Y)$, and let the number of $F(\mathbb{Z})$-equivalence classes of \textbf{negative-discriminant} monogenic binary cubic forms with height less than $Y$ be $N^-(Y)$. Then, 
$$N^+(Y) = \frac{8}{135}Y^{5/6} + O(Y^{1/2+\epsilon})$$
$$N^-(Y) = \frac{32}{135}Y^{5/6} + O(Y^{1/2+\epsilon})$$
\end{proposition}

Figure 1 below shows $N^+(Y)$ in a dotted blue line (computed by our program) and $\frac{8}{135}Y^{5/6}$ in a solid red line. Figure 2 shows $N^-(Y)$ in a dotted blue line (also computed by our program) and $\frac{32}{135}Y^{5/6}$ in a solid red line. In Figure 1, $Y$ goes up to $10^{11}$; in Figure 2, $Y$ goes up to $10^{10}$. 

\begin{figure}[H]
\begin{center}
\pgfplotstableread{bcf_gen_pos.txt}
	\datatable
\begin{tikzpicture}[scale = 1.15]
\begin{axis}[xlabel={$Y$},
xlabel style = {font=\tiny,xshift=0.5ex},
scaled y ticks = false, 
yticklabel style = {font=\fontsize{7}{4}\selectfont,xshift=0.5ex},
scaled x ticks = false,
xticklabel style = {font=\fontsize{7}{4}\selectfont,yshift=0.5ex}, 
legend pos = north west, 
xticklabels = {$0$, $1$, $2 \cdot 10^{10}$, $4 \cdot 10^{10}$, $6 \cdot 10^{10}$, $8 \cdot 10^{10}$, $1 \cdot 10^{11}$}]
\addplot[color=blue,very thick, dotted]
	table[y = BCF+]{\datatable};
\addlegendentry{\fontsize{5}{4}\selectfont{$N^+(Y)$}};
\addplot[color=red]
	table[y = Theoretical]{\datatable};
\addlegendentry{\fontsize{5}{4}\selectfont{$\frac{8}{135}Y^{5/6}$}};
\end{axis}
\end{tikzpicture}

\caption{\centering \footnotesize The number of positive-discriminant monogenic binary cubic forms bounded by height $Y$, compared to the equation $\frac{8}{135}Y^{5/6}$}
\end{center}
\end{figure}
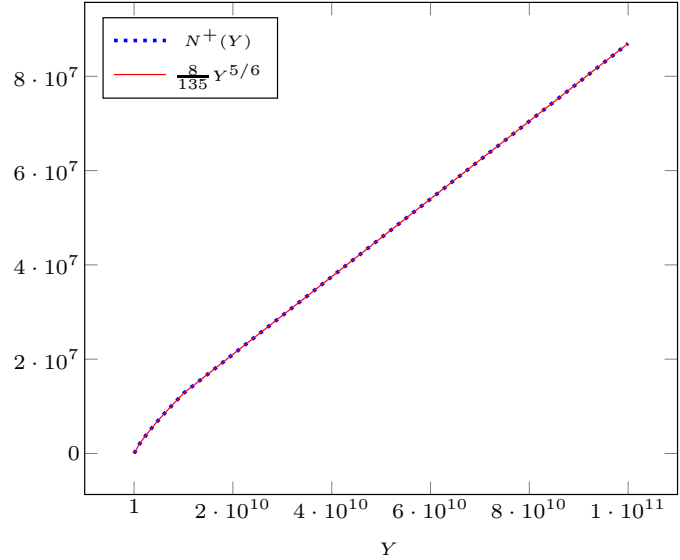

\begin{figure}[H]
\begin{center}
\pgfplotstableread{bcf_gen_neg.txt}
	\datatable
\begin{tikzpicture}[scale = 1.1]
\begin{axis}[xlabel={$Y$},
xlabel style = {font=\tiny,xshift=0.5ex},
scaled y ticks = false, 
yticklabel style = {font=\fontsize{7}{4}\selectfont,xshift=0.5ex},
scaled x ticks = false, 
xticklabel style = {font=\fontsize{7}{4}\selectfont,yshift=0.5ex}, 
legend pos = north west, 
xticklabels = {$0$, $1$, $2 \cdot 10^9$, $4 \cdot 10^9$, $6 \cdot 10^9$, $8 \cdot 10^9$, $1 \cdot 10^{10}$}]
\addplot[color=blue, very thick, dotted]
	table[y = BCF-]{\datatable};
\addlegendentry{\fontsize{5}{4}\selectfont{$N^-(Y)$}};
\addplot[color=red]
	table[y = Theoretical]{\datatable};
\addlegendentry{\fontsize{5}{4}\selectfont{$\frac{32}{135}Y^{5/6}$}};
\end{axis}
\end{tikzpicture}
\caption{\centering \footnotesize The number of negative-discriminant monogenic binary cubic forms bounded by height $Y$, compared to the equation $\frac{32}{135}Y^{5/6}$}
\end{center}
\end{figure}
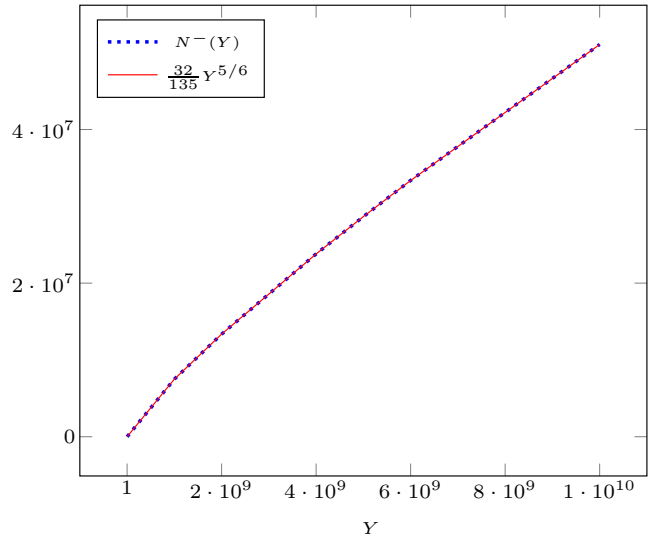

Note that in both Figure 1 and Figure 2, from a visual standpoint, the theoretical and computational results are in agreement. Below is a more detailed error analysis. 

For $Y=10^{11}$, our program calculated that $N^+(Y) = 86,961,377$. The expected theoretical value of $N^+(Y)$ is $\frac{8}{135}(10^{11})^{5/6} = 86,980,697.341$, implying our calculated value has an error of $0.022\%$.

For $Y=10^{10}$, our program calcualted that $N^-(Y) = 51,074,450$. The expected theoretical value of $N^-(Y)$ is $\frac{32}{135}(10^{10})^{5/6} = 51,068,081.542$, implying our calculated value has an error of $0.012\%$. 

From these low error percentages, it is clear that our computational results quickly agree with the corresponding theoretical results. 

\subsection{Asymptotics on Binary Cubic Form Maximality Ratios}

The following proposition can be deduced from Proposition \ref{bcf-counts} above, in conjunction with Theorem 3.6 of \cite{BHS}. 

\begin{proposition}
Let $N_{max}^{+}(Y)$ be the number of $F(\mathbb{Z})$-equivalence classes of \textbf{maximal} \textbf{positive-discriminant} binary cubic forms with height less than $Y$, and let $N_{max}^{-}(Y)$ be the number of $F(\mathbb{Z})$-equivalence classes of \textbf{maximal} \textbf{negative-discriminant} binary cubic forms with height less than $Y$. In addition, let $N^+(Y)$, $N^-(Y)$ be as defined in Proposition \ref{bcf-counts}. Then, 
$$\frac{N_{max}^{+}(Y)}{N^{+}(Y)}=\frac{N_{max}^{-}(Y)}{N^{-}(Y)}=\frac{1}{\zeta(2)}$$
\end{proposition}

Figure 3 below demonstrates the convergence of both $\frac{N_{max}^{+}(Y)}{N^{+}(Y)}$ and $\frac{N_{max}^{-}(Y)}{N^{-}(Y)}$ to $\frac{1}{\zeta(2)}$.

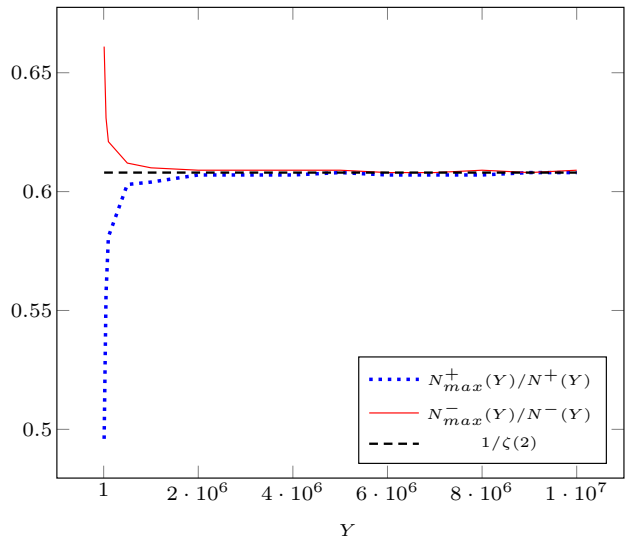
\begin{figure}[H]
\begin{center}
\pgfplotstableread{max_ratios.txt}
	\datatable
\begin{tikzpicture}[scale = 1.1]
\begin{axis}[xlabel={$Y$},
xlabel style = {font=\tiny,xshift=0.5ex},
scaled y ticks = false, 
yticklabel style = {font=\fontsize{7}{4}\selectfont,xshift=0.5ex},
scaled x ticks = false,
xticklabel style = {font=\fontsize{7}{4}\selectfont,yshift=0.5ex}, 
legend pos = south east, 
xticklabels = {$0$, $1$, $2 \cdot 10^6$, $4 \cdot 10^6$, $6 \cdot 10^6$, $8 \cdot 10^6$, $1 \cdot 10^7$}]
\addplot[color=blue, very thick, dotted]
	table[y = Max+]{\datatable};
\addlegendentry{\fontsize{5}{4}\selectfont{$N_{max}^{+}(Y) / N^{+}(Y)$}};
\addplot[color=red]
	table[y = Max-]{\datatable};
\addlegendentry{\fontsize{5}{4}\selectfont{$N_{max}^{-}(Y) / N^{-}(Y)$}};
\addplot[color=black, thick, densely dashed]
	table[y = Obj]{\datatable};
\addlegendentry{\fontsize{5}{4}\selectfont{$1 / \zeta(2)$}};
\end{axis}
\end{tikzpicture}
\caption{\centering \footnotesize The fractions of positive-discriminant and negative-discriminant monogenic binary cubic forms that are maximal, compared with $\frac{1}{\zeta(2)}$}
\end{center}
\end{figure}


In Figure 3, as early as $Y = 2 \cdot 10^{6}$ , we have 


$$\dfrac{N_{max}^{+}(Y)}{N^{+}(Y)} = \dfrac{6318}{10405} = 0.607$$

$$\dfrac{N_{max}^{-}(Y)}{N^{-}(Y)} = \dfrac{25662}{42143} = 0.609$$

Both of these values are within $0.2$ percent of $\frac{1}{\zeta(2)}$. 


In order to ensure that a maximality constraint on our binary cubic forms would not introduce too much additional error to our computation, let us perform an error analysis similar to the one we conducted at the end of Section 4.1, with the requirement that our binary cubic forms be maximal. 

For $Y = 2 \cdot 10^{6}$, our program calculated that $N_{max}^{+}(Y) = 6,318$. The expected theoretical value of $N_{max}^{+}(Y)$ is $\frac{8}{135\zeta(2)}(2 \cdot 10^{6})^{5/6} = 6,418.980$, implying our calculated value has an error of $1.573\%$.

For the same value of $Y$, our program calculated that $N_{max}^{-}(Y) = 25,662$. The expected theoretical value of $N_{max}^{-}(Y)$ is $\frac{32}{135\zeta(2)}(2 \cdot 10^{6})^{5/6} = 25,675.922$, implying our calculated value has an error of $0.043\%$. 

Since these error percentages are once again low, we are assured that our computational results agree with the corresponding theoretical predictions. 

\subsection{Asymptotic Averages of Binary Cubic Forms}

We are now ready to give evidence towards Theorem \ref{BHSthm}. The computational verification process relied heavily on Eureqa, a computer application that employs genetic programming \cite{K} to find regression curves that optimally fit the data given. See section 5.1 for more details on how Eureqa works.

Figures 4 and 5 below display $\mu_2(\mathcal{F}^{+}(Y))$ and $\mu_2(\mathcal{F}^{-}(Y))$ with respect to $Y$: 


\begin{figure}[H]
\begin{center}
\pgfplotstableread{pos_avgs2.txt}
	\datatable
{
\begin{tikzpicture}[scale = 1.1]
\begin{axis}[xlabel={$Y$},
xlabel style = {font=\tiny,xshift=0.5ex},
scaled x ticks = false,
xticklabel style = {font=\fontsize{7}{4}\selectfont,yshift=0.5ex}, 
ylabel={$\mu_2(\mathcal{F}^{+}(Y))$},
ylabel style = {font=\tiny,xshift=0.5ex},
scaled y ticks = false, 
yticklabel style = {font=\fontsize{7}{4}\selectfont,xshift=0.5ex},
legend pos = north west, 
xticklabels = {$0$, $1$, $2 \cdot 10^{10}$, $4 \cdot 10^{10}$, $6 \cdot 10^{10}$, $8 \cdot 10^{10}$, $1 \cdot 10^{11}$}]
\addplot[color=blue]
	table[y = N2]{\datatable};
\addplot[color=black, domain = 0:100000000000]{1.25};
\addplot[color=red, domain = 0:100000000000]{1.5};
\node[circle,fill,inner sep=1.5pt] at (axis cs:787702088,1.25) {};
\node[label={270:{\scriptsize{$Y = 787,702,088$}}}] at (axis cs:21000000000,1.26) {}; 
\end{axis}
\end{tikzpicture}
}
\caption{\centering \footnotesize $\mu_2(\mathcal{F}^{+}(Y))$, compared to the general binary cubic form asymptote in \cite{B} (in black), as well as the monogenic binary cubic form asymptote predicted in \cite{BHS} (in red)}
\end{center}
\end{figure}

\begin{figure}[H]
\centering
\pgfplotstableread{neg_avgs.txt}
	\datatable
{
\begin{tikzpicture}[scale = 1.1]
\begin{axis}[xlabel={$Y$},
xlabel style = {font=\tiny,xshift=0.5ex},
scaled x ticks = false, 
xticklabel style = {font=\fontsize{7}{4}\selectfont,yshift=0.5ex}, 
ylabel={$\mu_2(\mathcal{F}^{-}(Y))$},
ylabel style = {font=\tiny,xshift=0.5ex},
scaled y ticks = false, 
yticklabel style = {font=\fontsize{7}{4}\selectfont,xshift=0.5ex},
legend pos = north west, 
xticklabels = {$0$, $1$, $2 \cdot 10^9$, $4 \cdot 10^9$, $6 \cdot 10^9$, $8 \cdot 10^9$, $1 \cdot 10^{10}$}]
\addplot[color=blue]
	table[y = N2]{\datatable};
\addplot[color=black, domain = 0:10000000000]{1.5};
\addplot[color=red, domain = 0:10000000000]{2};
\node[circle,fill,inner sep=1.5pt] at (axis cs:17382351,1.5) {};
\node[label={270:{\scriptsize{$Y = 17,382,351$}}}] at (axis cs:2000000000,1.525) {};
\end{axis}
\end{tikzpicture}
}
\caption{\centering \footnotesize $\mu_2(\mathcal{F}^{-}(Y))$, compared to the general binary cubic form asymptote in \cite{B} (in black), as well as the monogenic binary cubic form asymptote predicted in \cite{BHS} (in red)}
\end{figure} 

Note that upon examination of the above two graphs, we can instantly determine that $\lim_{Y \rightarrow \infty} \mu_2(\mathcal{F}^{+}(Y))$ and $\lim_{Y \rightarrow \infty} \mu_2(\mathcal{F}^{-}(Y))$ are different from the asymptotic averages in Theorem \ref{Bthm} (where no monogenicity condition is imposed). This is undoubtedly a surprising result. 

After we inserted all of the data comprising the above two graphs into a spreadsheet, Eureqa formulated the following pair of equations nearly instantaneously:


\[
\begin{cases}
\mu_2(\mathcal{F}^{+}(Y)) \approx 1.51-1.34Y^{-0.0810} \\ 
\mu_2(\mathcal{F}^{-}(Y)) \approx 1.99-2.05Y^{-0.0878}
\end{cases}
\]

These equations have the exact form we are looking for: asymptotes of approximately $3/2$ and $2$, respectively, followed by second-order terms that vanish to zero as $Y$ approaches infinity. It is also worth noting that, when compared against the data, both of the above equations have mean-squared error (MSE) values below $10^{-7}$.

It is evident now that this model independently relates the computational data to the asymptotics given in Theorem \ref{BHSthm}. The next step, as mentioned previously, is to extend our approach for computing class group $2$-torsion sizes to computing class group $p$-torsion sizes, where $p$ is a prime integer. This extension is detailed in the next section, along with our main results. 

\section{Main Results}



\subsection{Introduction to Genetic Programming}


As mentioned in Section 4.3, Eureqa \cite{SL} is a computer application that uses genetic programming \cite{K} to find regression curves that optimally fit the data given. Here we will discuss what genetic programming is, and why it is useful for solving regression problems like the ones we are working with. 

Genetic programming (GP) is an artificial intelligence-based problem-solving technique. Broadly speaking, if a GP-based algorithm is given a problem to solve, it will start by developing randomly generated solutions that often do not solve the problem well. It will then conduct two types of operations, namely ``reproduction'' operations and ``mutation'' operations, to try to modify the original solutions and develop new ones. A ``reproduction'' operation will take two proposed solutions and combine them together somehow to form a new solution. A ``mutation'' operation will take a proposed solution and change it in some way.\footnote{It may be obvious by now that the name ``genetic programming,'' as well as the names of the two types of operations involved in GP, are derived from Darwin's theory of evolution.}

Through combinations of these two types of operations, the algorithm will improve the original randomly-generated solutions and develop new solutions that solve the problem better. We are able to discern whether or not a proposed solution solves the given problem well based on some evaluation metric (e.g. mean-squared error) – this metric is usually referred too as a loss function. Generally, a solution with a lower loss (i.e. a lower loss-function value) is a better solution. 

In our case, we are using a GP-based algorithm – i.e. Eureqa – to solve a regression problem. We start with a set of random solutions, and develop better solutions by conducting ``reproduction'' and ``mutation'' operations on the original solutions. We determine which solutions are better by looking at their mean-squared errors with respect to the original data (i.e. the $\mu_p(\mathcal{F}^{\pm}(Y))$ values that we provide Eureqa, for fixed $p$ and increasing $Y$). It is worth noting, however, that we often do not choose the solutions with the best mean-squared errors, because these solutions are often so complex that they over-fit the given data. In other words, they have poor predictive power for high $Y$-values that are not supplied to Eureqa due to computational limitations. Thus, we tend to look for solutions with sufficiently low mean-squared errors whose forms are as simple as possible. These solutions often take the same form as the $\mu_2(\mathcal{F}^{\pm}(Y))$ equations stated at the end of Section 4.3: 

$$\mu_p(\mathcal{F}^{\pm}(Y)) \approx \alpha^{\pm}-\beta^{\pm}Y^{\gamma^{\pm}}$$ where $\alpha^{\pm},\beta^{\pm},\gamma^{\pm} \in \mathbb{R}$.

Note that GP-based algorithms are often highly probabilistic, so running the same program several times in a row may yield multiple different answers. When comparing the results presented at the end of Section 5.3 with the conjectures stated in Section 6, bear in mind that any differences between computed and conjectured values of $\lim_{Y \rightarrow \infty}\mu_p(\mathcal{F}^{\pm}(Y))$ can be at least partially attributed to the noise inherent to Eureqa's training algorithm.

\subsection{Eureqa Methodology}

In principle, the approach for calculating class group $p$-torsion sizes is highly similar to the approach for calculating class group $2$-torsion sizes. The only theoretical difference lies in Step 5 of the methodology detailed in Section 3. 
By Lemma \ref{TTS}, all we have to do is find all the elements of $Cl(K)[p]$.s

However, due to Proposition \ref{bcf-counts}, when counting monogenic binary cubic forms bounded by height $Y$, there are far more negative-discriminant binary cubic forms than positive-discriminant binary cubic forms. Thus, it is easier to push the positive-discriminant height bound to a high value than the negative-discriminant height bound. In our computation, we pushed the positive-discriminant height bound to $10^{11}$, and we pushed the negative-discriminant height bound to $10^{10}$. 

As we saw in Section 4.3, the formulas Eureqa discovered were of the form $$\mu_p(\mathcal{F}^{\pm}(Y)) \approx \alpha_{i}^{\pm}-\beta_{i}^{\pm}Y^{\gamma_{i}^{\pm}}$$ where $\alpha_{i}^{\pm},\beta_{i}^{\pm},\gamma_{i}^{\pm} \in \mathbb{R}$. It is safe to assume that $\gamma_{i}^+ = \gamma_{i}^-$ \cite{R}. Thus, since we have a higher height bound for positive-discriminant binary cubic forms than for negative-discriminant binary cubic forms, we can give $\gamma_{i}^+$ to the negative-discriminant Eureqa model as a fixed constant. In other words, we can set $\gamma_{i}^-$ to always be equal to $\gamma_{i}^+$. 

With this in mind, we can develop a new Eureqa strategy. Let $p$ be a prime number. First, we train the positive-discriminant model as usual. This yields our formula for $\mu_p(\mathcal{F}^{+}(Y))$, which has the form $\alpha_{i}^{+}-\beta_{i}^{+}Y^{\gamma_{i}^{+}}$. We then train a \textit{second} positive discriminant model where Eureqa is told to use the form $\alpha_{f}^{+}-\beta_{f}^{+}Y^{\gamma_{f}^{+}}$. In addition, the $\alpha_{f}^+$ value – i.e. the asymptote – is fixed. This way, we can eliminate some of the unwanted noise inherent to Eureqa's training algorithm and focus on finding a more accurate value of the exponent $\gamma_{f}^+$. (This is discussed in Section 5.1 above.) We refer to this process as ``fine-tuning" the exponent. 


After all of this is finished, we begin training the negative model. We mandate that Eureqa uses the form $\alpha_{f}^{-}-\beta_{f}^{-}Y^{\gamma_{f}^{-}}$, and we set $\gamma_{f}^-$ to be equal to $\gamma_{f}$, i.e. the fine-tuned version of $\gamma_{i}^+$. After Eureqa discovers $\alpha_{f}^-$ and $\beta_{f}^-$ for us, we have our formula for $\mu_p(\mathcal{F}^{-}(Y))$.\footnote{Note that we use the $\gamma_{i}^{+}$ exponents to state our $\mu_p(\mathcal{F}^{+}(Y))$ formulas, and we use the $\gamma_{f}^{-}$ exponents to state our $\mu_p(\mathcal{F}^{-}(Y))$ formulas. Thus, for a given $p$, the exponent in the $\mu_p(\mathcal{F}^{+}(Y))$ expression will not necessarily be the same as the exponent in the $\mu_p(\mathcal{F}^{-}(Y))$ expression.}

\subsection{Results}

Below are the graphs that we obtained for the average class group $p$-torsion sizes – the first graph is for positive-discriminant binary cubic forms, and the second is for negative-discriminant binary cubic forms. 

\begin{figure}[H]
\begin{center}
\pgfplotstableread{pos_avgs2.txt}
	\datatable
\begin{tikzpicture}[scale = 1.1]
\begin{axis}[xlabel={$Y$},
xlabel style = {font=\tiny,xshift=0.5ex},
scaled x ticks = false, 
xticklabel style = {font=\fontsize{7}{4}\selectfont,yshift=0.5ex}, 
ylabel={$p$-Torsion Size},
ylabel style = {font=\tiny,xshift=0.5ex},
scaled y ticks = false, 
yticklabel style = {font=\fontsize{7}{4}\selectfont,xshift=0.5ex},
legend pos = north west, 
xticklabels = {$0$, $1$, $2 \cdot 10^{10}$, $4 \cdot 10^{10}$, $6 \cdot 10^{10}$, $8 \cdot 10^{10}$, $1 \cdot 10^{11}$}]
\addplot[color=blue]
	table[y = N2]{\datatable};
\addplot[color=red]
	table[y = N3]{\datatable};
\addplot[color=black]
	table[y = N5]{\datatable};
\addplot[color=green]
	table[y = N7]{\datatable};
\addplot[color=orange]
	table[y = N11]{\datatable};
\end{axis}
\end{tikzpicture}
\caption{\centering \footnotesize $\mu_p(\mathcal{F}^{+}(Y))$ for $p = \textcolor{blue}{2}, \textcolor{red}{3}, 5, \textcolor{green}{7}, $ and $\textcolor{orange}{11}$}
\end{center}
\end{figure}

\begin{figure}[H]
\begin{center}
\pgfplotstableread{neg_avgs.txt}
	\datatable
\begin{tikzpicture}[scale = 1.1]
\begin{axis}[xlabel={$Y$},
xlabel style = {font=\tiny,xshift=0.5ex},
scaled x ticks = false, 
xticklabel style = {font=\fontsize{7}{4}\selectfont,yshift=0.5ex}, 
ylabel={$p$-Torsion Size},
ylabel style = {font=\tiny,xshift=0.5ex},
scaled y ticks = false, 
yticklabel style = {font=\fontsize{7}{4}\selectfont,xshift=0.5ex},
legend pos = north west, 
xticklabels = {$0$, $1$, $2 \cdot 10^9$, $4 \cdot 10^9$, $6 \cdot 10^9$, $8 \cdot 10^9$, $1 \cdot 10^{10}$}]
\addplot[color=blue]
	table[y = N2]{\datatable};
\addplot[color=red]
	table[y = N3]{\datatable};
\addplot[color=black]
	table[y = N5]{\datatable};
\addplot[color=green]
	table[y = N7]{\datatable};
\addplot[color=orange]
	table[y = N11]{\datatable};
\end{axis}
\end{tikzpicture}
\caption{\centering \footnotesize $\mu_p(\mathcal{F}^{-}(Y))$ for $p = \textcolor{blue}{2}, \textcolor{red}{3}, 5, \textcolor{green}{7}, $ and $\textcolor{orange}{11}$.}
\end{center}
\end{figure}

Note that at $Y = 10^{11}$, 

\[
\begin{cases}
\mu_2(\mathcal{F}^{+}(10^{11})) = 1.333 \\
\mu_3(\mathcal{F}^{+}(10^{11})) = 1.259 \\ 
\mu_5(\mathcal{F}^{+}(10^{11})) = 1.039 \\ 
\mu_7(\mathcal{F}^{+}(10^{11})) = 1.020 \\ 
\mu_{11}(\mathcal{F}^{+}(10^{11})) = 1.008
\end{cases}
\]

Also note that at $Y = 10^{10}$, 

\[
\begin{cases}
\mu_2(\mathcal{F}^{-}(10^{10})) = 1.714 \\
\mu_3(\mathcal{F}^{-}(10^{10})) = 1.645 \\ 
\mu_5(\mathcal{F}^{-}(10^{10})) = 1.203 \\ 
\mu_7(\mathcal{F}^{-}(10^{10})) = 1.142 \\ 
\mu_{11}(\mathcal{F}^{-}(10^{10})) = 1.090
\end{cases}
\]

Using the new Eureqa strategy described in Section 5.2, we obtained the following generalized asymptotic average predictions. Note that we used mean-squared error (MSE) as our error metric, since it is commonly used in the machine learning community. 


\textbf{Positive Discriminants: }
\[
\begin{cases}
\mu_2(\mathcal{F}^{+}(Y)) \approx 1.51-1.34Y^{-0.0810} \longrightarrow \text{MSE} = 5.620 \times 10^{-9} \\
\mu_3(\mathcal{F}^{+}(Y)) \approx  1.30-1.72Y^{-0.143} \ \longrightarrow \text{MSE} = 1.106 \times 10^{-7}\\ 
\mu_5(\mathcal{F}^{+}(Y)) \approx 1.04-0.595Y^{-0.196} \ \longrightarrow \text{MSE} = 8.359 \times 10^{-9}\\ 
\mu_7(\mathcal{F}^{+}(Y)) \approx 1.02-0.248Y^{-0.171} \ \longrightarrow \text{MSE} = 7.595 \times 10^{-9}\\ 
\mu_{11}(\mathcal{F}^{+}(Y)) \approx 1.01-0.548Y^{-0.261} \ \longrightarrow \text{MSE} = 3.108 \times 10^{-9}
\end{cases}
\]


\textbf{Negative Discriminants\footnote{Note that the formula for $\mu_2(\mathcal{F}^{-}(Y))$ is different from the one in Section 4.3 – this is because we recalculated it using the new Eureqa strategy from Section 5.2.} : }
\[
\begin{cases}
\mu_2(\mathcal{F}^{-}(Y)) \approx 2.00-1.95Y^{-0.0832} \ \longrightarrow \text{MSE} = 1.746 \times 10^{-7}\\
\mu_3(\mathcal{F}^{-}(Y)) \approx 1.45+0.02Y^{0.0939} \ \longrightarrow \text{MSE} = 5.936 \times 10^{-5}\\
\mu_5(\mathcal{F}^{-}(Y)) \approx 1.24-0.254Y^{-0.0872} \ \longrightarrow \text{MSE} = 2.349 \times 10^{-6}\\
\mu_7(\mathcal{F}^{-}(Y)) \approx 1.16-0.485Y^{-0.145} \ \longrightarrow \text{MSE} = 2.995 \times 10^{-6}\\
\mu_{11}(\mathcal{F}^{-}(Y)) \approx 1.10-0.643Y^{-0.178} \ \longrightarrow \text{MSE} = 1.138 \times 10^{-6}
\end{cases}
\]


\section{Conjectures on Class Group $p$-Torsion Sizes}

Based on the set of Eureqa-predicted asymptotes from Section 5, we can now formulate conjectures on the general form of $\lim_{Y \rightarrow \infty}\mu_p(\mathcal{F}^{+}(Y))$ and $\lim_{Y \rightarrow \infty}\mu_p(\mathcal{F}^{-}(Y))$ for $p$ prime, $p \neq 3$.\footnote{Theoretical results indicate that we should not include $p=3$ in the statement of our conjecture \cite{CL}.} The following conjectures were also predicted by Eureqa. 

\begin{conjecture} \label{pos-conjecture} Let $p$ be a prime number such that $p \neq 3$. Then, 
$$\lim_{Y \rightarrow \infty}\mu_p(\mathcal{F}^{+}(Y)) = 1+\dfrac{1}{p(p-1)}$$
\end{conjecture}

\begin{conjecture} \label{neg-conjecture} Let $p$ be a prime number such that $p \neq 3$. Then, 
$$\lim_{Y \rightarrow \infty}\mu_p(\mathcal{F}^{-}(Y)) = 1+\dfrac{1}{p-1}$$
\end{conjecture}

These conjectures predict most of the asymptotes in Section 5 with minimal error. The one pair of exceptions is $\lim_{Y \rightarrow \infty}\mu_3(\mathcal{F}^{\pm}(Y))$ – however, note that these conjectures are not expected to hold for $p=3$ \cite{CL}. 

\begin{remark}
It is worth noting that according to conjectures \ref{pos-conjecture} and \ref{neg-conjecture}, $$\lim_{Y \rightarrow \infty}\mu_p(\mathcal{F}^{-}(Y)) - \lim_{Y \rightarrow \infty}\mu_p(\mathcal{F}^{+}(Y)) = \left(1+\dfrac{1}{p-1}\right) - \left(1+\dfrac{1}{p(p-1)}\right) = \dfrac{1}{p}$$
\end{remark}

\section{Conclusions}

Conjectures \ref{pos-conjecture} and \ref{neg-conjecture} are the main contributions of this paper. They are the culminating result of our overall methodology – computationally rediscovering the theoretical result in \cite{BHS} on the average $2$-torsion size of class groups of cubic fields, and then extending this result from average $2$-torsion sizes to average $p$-torsion sizes, for all prime $p$. It is worth noting, however, that these conjectures were made solely based on computational evidence. We hope that future work in this field presents theoretical evidence for these conjectures and provides deeper insight into what these conjectures truly mean. 

\section*{Acknowledgements}


Many thanks to Prof.~Ila Varma (University of Toronto), whose guidance throughout the course of this project was invaluable. Also, thanks to Prof.~Jon Hanke (Princeton University) and Dr.~Dylan Yott (UC Berkeley) for their mentorship during the early stages of this project, and to Prof.~Arul Shankar (University of Toronto) for his advice later in the project. Finally, thanks to Hari Pingali and Stephen New, whose contributions during the early stages of the project helped shape its future.

\end{document}